\documentclass[a4paper, 12pt, twoside]{article}
\usepackage[latin1]{inputenc}
\usepackage[T1]{fontenc}
\usepackage[english]{babel}
\usepackage{amsmath,amssymb,amscd}
\usepackage{amsthm}
\usepackage{array}
\usepackage{hyperref}
\usepackage{graphicx}
\usepackage[pdftex]{color}
\usepackage{enumerate}
\usepackage{enumitem}
\usepackage{cancel}
\usepackage{comment}
\usepackage{tikz}
\usetikzlibrary{graphs,arrows,shapes,positioning}
\newtheorem{definition}{Definition}[section]
\newtheorem{remark}[definition]{Remark}

\newtheorem{corollary}[definition]{Corollary}
\newtheorem{lemma}[definition]{Lemma}
\newtheorem{proposition}[definition]{Proposition}

\newtheorem{question}[definition]{Question}

\def\Ind#1#2{#1\setbox0=\hbox{$#1x$}\kern\wd0\hbox to 0pt{\hss$#1\mid$\hss}
\lower.9\ht0\hbox to 0pt{\hss$#1\smile$\hss}\kern\wd0}
\def\ind{\mathop{\mathpalette\Ind{}}}
\def\Notind#1#2{#1\setbox0=\hbox{$#1x$}\kern\wd0\hbox to 0pt{\mathchardef
\nn=12854\hss$#1\nn$\kern1.4\wd0\hss}\hbox to 
0pt{\hss$#1\mid$\hss}\lower.9\ht0
\hbox to 0pt{\hss$#1\smile$\hss}\kern\wd0}
\def\nind{\mathop{\mathpalette\Notind{}}}

\title{A rank based on Shelah trees}
\author{Santiago C\'ardenas-Mart\'in, Rafel Farr\'e}
\date{Version 2021-06-16}

\begin{document} 

\maketitle
\thispagestyle{empty}

\bigskip
\begin{abstract}
	We define a global rank for partial types based on a generalization of Shelah trees.
	We prove an equivalence with the depth of a localized version of the constructions known as dividing sequence 
	  and dividing chain.
	This rank characterizes simple and supersimple types. Moreover, this rank does not change for non-forking
	  extensions under certain hypothesys.
	We also prove this rank satisfies Lascar-style inequalities. 
\end{abstract}

\section{Conventions}
    We denote $L$ a language and $T$ a complete theory. 
    We denote by $\mathfrak{C}$ a monster model of $T$ and assume that it is $\kappa$-saturated and strongly
      $\kappa$-homogeneous for a cardinal $\kappa$ larger enough.
    Every set of parameters $A,B,\ldots$ is considered as a subset of $\mathfrak{C}$ with cardinal less than
      $\kappa$.
	We denote $Ord$ the class of Ordinals and $Lim$ the class of Limit Ordinals.
	
    We denote $a,b,\ldots$ tuples of elements of the monster model, possibly infinite (of length less than $\kappa$).
    We often use these tuples as ordinary sets regardless of their order. 
    We often omit union symbols, for example we write $ABc$ to mean $A\cup B\cup c$.
	Given a sequence of sets $( A_i : i\in\alpha )$ we use $A_{<i}$ and $A_{\leq i}$  to denote $\bigcup_{j<i}A_j$ and $\bigcup_{j\leq i}A_j$ respectively.
	We use $I$ to denote a infinite index set without order and use $O$ for a infinite lineal ordered set.
	Unless otherwise stated, all the complete types are finitary.
	We use $\ind^d$ and $\ind^f$ to denote the independence relations for non-dividing and non-forking respectively.
	By $dom(p)$ we denote the set of all parameters that appear in some formula of $p$.

\section{Definitions and basic properties}
    \begin{definition}\label{DDDef}
		Let $p(x)$ be a partial type.
		Let $\beta$ denote the supremum (if it exists) of all possible depths $\alpha$ for which there
		  exist a tree of parameters $(a_{s,i}:s<\omega^{<\alpha},i\in\omega)$, a sequence of formulas
		  ($\varphi_j(x,y_j)\in L:j\in\alpha)$ and a sequence of numbers $(k_j:j\in\alpha)$ such that 
			\begin{enumerate}[label=(\alph*)]
    		  \item For each $s\in\omega^{<\alpha}$, $\{\varphi_{|s|}(x,a_{s,i}):i<\omega\}$ is
    		    $k_{|s|}$-inconsistent.
    		  \item For each $f:\alpha\longrightarrow\omega$, 
    			$p(x)\cup\{\varphi_i(x,a_{f\upharpoonright i,f(i)}):i\in\alpha\}$ is consistent.
			\end{enumerate}
	  
        If there is not such tree for $p(x)$ we set $DD(p)=0_+$.	  
	    If this supremum does not exist we write $DD(p)=\infty$. 
		Otherwise, if $\beta$ is attained we put $DD(p)=\beta_+$ and $DD(p)=\beta_-$ if it is not attained.
		
		We call $DD(p)$ the \textbf{Dividing Depth of $p$}.
		As usual, we write $DD(a/A)$ for $DD(tp(a/A))$.
	\end{definition}

	We observe that if $\beta$ is a succesor ordinal or $0$, $DD$ can not take the value $\beta_-$.
	In this case, sometimes we will write the ordinal number without subscript.
    We order the class possible of values of $DD$ $\{\beta_-:\beta\in Lim\setminus\{0\}\}\cup\{\beta_+:\beta\in Ord\}\cup\{\infty\}$ in a natural way, if $\alpha<\beta$, then $\alpha_*<\beta_*$, and for every $\beta$, $\beta_-<\beta_+$ and $\beta_*<\infty$.

    It is important to observe that this definition does not depend of the choice of the set of parameters of the set of formulas.	
	Now, we take the usual constructions known as dividing sequence (see for example Tent-Ziegler\cite{Ziegler}) and
	  dividing chain (see for example Kim\cite{Kim}) and we localize them for a type $p$:
	\begin{definition} 
		Let $p(x)$ be a partial type over a set of parameters $A$.
		\begin{enumerate}
			\item A \textbf{dividing sequence of depth $\alpha$ in $p$ over $A$} consists in a sequence of formulas
			  $(\varphi_i(x,x_i)\in L:i\in\alpha)$ and a sequence of parameters $(a_i:i\in\alpha)$ such that 
				\begin{enumerate}
					\item $p(x)\cup\{\varphi_i(x,a_i):i\in\alpha\}$ is consistent.
					\item for each $i\in\alpha,\varphi_i(x,a_i)$ divides over $Aa_{<i}$.
				\end{enumerate}

			\item A \textbf{dividing chain of complete types of depth $\alpha$ in $p$ over $A$} is a sequence of
			  complete types $(p_i(x):i\in\alpha)$  such that $p\subseteq p_0$, $A\subseteq dom(p_0)$, $p_0$ divides
			  over $A$ and for every $0<i<\alpha$, $p_i$ is a dividing extension of	$p_{<i}$.

			\item A \textbf{dividing chain of partial types of depth $\alpha$ in $p$ over $A$} consist in a sequence
			  of partial types $(p_i(x):i\in\alpha)$ and a sequence of sets of parameters $(A_i:i\in\alpha)$, each
			  $p_i$ a partial type over $A_i$, $p\subseteq p_0$, $A\subseteq A_0$, $p_0$ divides over $A$ and for
			  every $0<i<\alpha$, $p_{<i}\subseteq p_i$, $A_{<i}\subseteq A_i$ and $p_i$ divides over $A_{<i}$.
	    \end{enumerate}
	\end{definition}

    \begin{remark}
      In the definition of dividing sequence, it is immediate to prove that it is equivalent consider the formulas 
        in $L$ or in $L(A)$.
    \end{remark}

	In what follows, $s\in \omega^{<\alpha}$ means that $s$ is a sequence of length less than $\alpha$ of elements of
	  $\omega$, i.e. a function from an ordinal less than $\alpha$ into $\omega$. 
	$|s|$ denotes the domain of the function $s$, so $|s|$ is the length of the sequence $s$.	

	\begin{proposition}\label{DDEquiv}
		Let $p(x)$ be a partial type over a set of parameters $A$ and $\alpha$ an ordinal.
		The following are equivalent:
		\begin{enumerate}
		    \item $DD(p)\geq\alpha_+$
		    
			\item There exist a tree of parameters $(a_{s,i}:s\in\omega^{<\alpha},i<\omega)$ and a sequence of
			  formulas $(\varphi_j(x,y_j)\in L:j\in\alpha)$ such that 
				\begin{enumerate}
					\item For each $s\in\omega^{<\alpha}$,$(a_{s,i}:i<\omega)$ is indiscernible over
					  $A\cup\{a_{s\upharpoonright j,s(j)}:j<|s|\}$.
					\item For each $s\in\omega^{<\alpha}$, $\{\varphi_{|s|}(x,a_{s,i}):i<\omega\}$ is inconsistent.
					\item For each $f:\alpha\longrightarrow\omega$, 
						$p(x)\cup\{\varphi_i(x,a_{f\upharpoonright i,f(i)}):i\in\alpha\}$ is consistent.
				\end{enumerate}

			\item There exist a tree of parameters $(a_{s,i}:s\in\omega^{<\alpha},i<\omega)$ and a tree of formulas 
			  $(\varphi_s(x,y_s)\in L:s\in\omega^{<\alpha})$ such that 
				\begin{enumerate}
					\item For each $s\in\omega^{<\alpha}$, $(a_{s,i}:i<\omega)$ is indiscernible over
					  $A\cup\{a_{s\upharpoonright j,s(j)}:j<|s|\}$.
					\item For each $s\in\omega^{<\alpha}$, $\{\varphi_s(x,a_{s,i}):i<\omega\}$ is inconsistent.
					\item For each $f:\alpha\longrightarrow\omega$, 
					  $p(x)\cup\{\varphi_{f\upharpoonright i}(x,a_{f\upharpoonright i,f(i)}):i\in\alpha\}$ is
					  consistent.
				\end{enumerate}

			\item There exists a dividing sequence in $p(x)$ of depth $\alpha$ over $A$.

			\item There exists a dividing chain of complete types in $p(x)$ of depth $\alpha$ $over A$.

			\item There exists a dividing chain of partial types in $p(x)$ of depth $\alpha$ over $A$.

		\end{enumerate}
	\end{proposition}
	\begin{proof} \
		\begin{itemize}
			\item[$1\Rightarrow 2$] 
				We fix a tree of depth $\alpha$ satisfying \emph{(a)}, \emph{(b)} of Definition~\ref{DDDef} and we
				  will construct by induction a  new tree replacing the parameters and preserving the same formulas
				  satisfying \emph{(a)} of \emph{2}, in addition to  \emph{(a)}, \emph{(b)} of 
				  Definition~\ref{DDDef}.
				To start, for each $i\in\omega$ we consider:
				  \[A_i=\left(a_{\emptyset,i}\right)^{\frown}
				    \left(a_{s,j}:s\in\omega^{<\alpha},s(0)=i,j<\omega\right)\]
                By the standard lemma (see for example lemma 7.1.1 in Tent-Ziegler\cite{Ziegler}) there is a sequence
                  $(A'_i:i\in\omega)$ indiscernible over $A$ and satisfying  $EM((A_i:i\in\omega)/A)$, 
                  the Ehrenfeucht-Mostowski type of $(A_i:i\in\omega)$ over $A$. 
                Obviously  $(a'_{\emptyset,i}:i\in\omega)$ is indiscernible over $A$.
                As the requirements \emph{(a)} and \emph{(b)} of Definition~\ref{DDDef} are in the
                  Ehrenfeucht-Mostowski type, the new parameters satisfy them too.
    			
                We assume we have a tree satisfying the Definition~\ref{DDDef} and moreover \emph{(a)} of \emph{2}
                  for $s\in\omega^{<\beta}$ with $\beta<\alpha$.
                We will obtain another one satisfying Definition~\ref{DDDef} and \emph{(a)} for $s\in\omega^\beta$.
    			For each $t\in\omega^\beta$ and $i\in\omega$ we consider:
    			  \[A_{t,i}=\left(a_{t,i}\right)^{\frown}
    				\left(a_{s,j}:s\in\omega^{<\alpha},s\upharpoonright\beta=t,s(\beta)=i,j<\omega\right)\]
    			As before we apply the standard lemma and obtain a sequence of subtrees $(A'_{t,i}:i\in\omega)$
    			  indiscernible over $A\{a_{t\upharpoonright\gamma,t(\gamma)}:\gamma<\beta\}$ and satisfying the
    			  Ehrenfeucht-Mostowski type
    			  $EM((A_{t,i}:i\in\omega)/A\{a_{t\upharpoonright\gamma,t(\gamma)}:\gamma<\beta\})$.

			\item[$2\Rightarrow 3$] 
				It is trivial, because the first tree is a special case of the second tree.

			\item[$3\Rightarrow 4$] 
				Every branch provides a dividing sequence.

			\item[$4\Rightarrow 5$]
				If $(\varphi_i(x,x_i)\in L:i\in\alpha)$ and  $(a_i:i\in\alpha)$ are a dividing sequence, consider
				  $A_i=Aa_{\leq i}$ and $a\models p\cup\{\varphi_i(x,a_i):i\in\alpha\}$.
				Then $(tp(a/A_i):i\in\alpha)$ is a dividing chain of depth $\alpha$ in $p$.
					
			\item[$5\Rightarrow 1$]
				Use, for example, proposition $3.8$ of Casanovas\cite{Casanovas11}.

			\item[$5\Rightarrow 6$] 
				Immediate.

			\item[$6\Rightarrow 4$] 
				Let $(p_i:i\in\alpha)$ and  $(A_i:i\in\alpha)$ be as in the definition.
				Let $a\models p_{<\alpha}$.
				For each $i\in\alpha$ we can take $\varphi_i(x,a_i)$ that divides over $A_{<i}$ and
				  $p_i\vdash\varphi_i(x,a_i)$.
				We may assume that $a_i\in A_i$. 
				So for each $i\in\alpha$, $\models\varphi_i(a,a_i)$, therefore
				  $p(x)\cup\{\varphi_i(x,a_i):i\in\alpha\}$ is consistent and $\varphi_i(x,a_i)$ divides over
				  $Aa_{<i}$.
		\end{itemize}
	\end{proof}

	\begin{proposition}\label{basicPropDD}
		Let $p(x),q(x)$ be partial types.
		Then,
		\begin{enumerate}
			\item If $f$ is an automorphism in the monster model, then $DD(p)=DD(p^f)$.
			\item If $p\vdash q$ then $DD(p)\leq DD(q)$
			\item $DD(p\vee q)=\max(DD(p),DD(q))$.
			\item $DD(p)=0$ if and only if $p$ is algebraic.
			\item Let $S,T$ be type-definable sets. 
				$DD(T)$ is well-defined by property $2$.
				Let $f:S\longrightarrow T$ be a definable bijection. 
				Then, $DD(T)=DD(S)$. 
		\end{enumerate}
	\end{proposition}
	\begin{proof}
		Without loss of generality, we can assume that $p$ and $q$ are over the same set of parameters $A$.
		\begin{enumerate}
			\item Immediate.

			\item If $p\vdash q$, any dividing sequence in $p$ is a dividing sequence in $q$.
			
			\item $DD(p\vee q)\geq\max(DD(p),DD(q))$ is immediate by $2$.
			  For the other inequality it suffices to remark that a dividing sequence in $p\vee q$ must be either a
			    dividing sequence in $p$ or a dividing sequence in $q$.

			\item $DD(p)\geq1$ means that there exists a formula $\varphi(x,a)$ consistent with $p$ such that 
  			    $\varphi$ divides over $A$.
  			    If $p$ is algebraic, for any realization $b$ of $p\cup\{\varphi\}$, $b\in acl(A)$. 
				But then  $b\ind^d_A a$, which is a  contradiction.

				On the other hand, if $p$ is not algebraic, we pick a realization $b$ of $p$, $b$ not algebraic over
				  $A$.
				Then, the formula $x=b$ divides over $A$ and is consistent with $p$, so $DD(p)\geq1$.

			\item We are going to prove that if $DD(S)\geq\alpha$ then $DD(T)\geq\alpha$.
								
  			  Let $S,T$ be defined by the types $p,q$ respectively.
			  Let $f:S\rightarrow T$ be defined by $\varphi(x,y)$.
			  We can assume that $\varphi(x,y),p(x),q(y)$ are all  over $A$.
			  Let $\alpha_+\leq DD(p)$ and let $s=\{\varphi_i(x,a_i):i\in\alpha\}$ be a dividing sequence in $p$ over
			    $A$.
			  Let $\psi_i(y,a_i)=\exists x(\varphi_i(x,a_i)\wedge\varphi(x,y))$.
			  We are going to prove that $t=\{\psi_i(y,a_i):i\in\alpha\}$ is a dividing sequence in $q$ over $A$.

			  We first check the consistency of $q\cup t$. 
			  Let $a\models p\cup s$ and $b=f(a)$.
			  Then, $b\models q\cup t$.
			  And now, we will verify that for any $i\in\alpha$, $\psi(y,a_i)$ divides over $Aa_{<i}$.
			  As $\varphi(x,a_i)$ divides over $Aa_{<i}$, we have a sequence $(c_{ij}:j\in\omega)$, indiscernible
			    over $Aa_{<i}$, such that $c_{ij}\equiv_{Aa_{<i}}a_i$ for every $j\in\omega$ and
			    $\{\varphi_i(x,c_{ij}):j\in\omega\}$ is inconsistent.
			  We will check that $(\psi_i(y,c_{ij}):j\in\omega)$ is inconsistent too. 
			  Assume that $(\psi_i(y,c_{ij}):j\in\omega)$ is consistent.
			  Then there exists $d\models\{\exists x(\varphi_i(x,c_{ij})\wedge\varphi(x,y)):j\in\omega\}$.
			  As $f$ is bijective, $x=f^{-1}(d)$ is the same for every $j$, so $f^{-1}(d)$ realizes
			    $\{\varphi_i(x,c_{ij}):j\in\omega\}$, a contradiction.
		\end{enumerate}
	\end{proof}

	\begin{proposition}\label{DDInf}
		If $DD(p)\geq(|T|^+)_+$ then $DD(p)=\infty$.
	\end{proposition}
	\begin{proof} 
		Assume that there exists a tree in the definition of $DD$ of depth $\alpha=|T|^+$.
		So, some formula $\varphi$ and a number $k$  appear together infinitely many times in the tree. 
		Therefore we can obtain a subtree with the same $\varphi$ and $k$ of depth $\omega$. 
		By a compactness argument we can obtain a tree of any depth.
	\end{proof}

	\begin{remark}\label{DDsupCT}
		Let $p(x)$ be a partial type over $A$.
		By equivalence $4$ of Proposition~\ref{DDEquiv}, it is immediate that $DD(p)=\sup\{DD(a/A):a\models p\}$.
		Therefore, if $DD(p)\geq\alpha_+$, there exists a completion $q\in S(A)$ of $p$ with $DD(q)\geq\alpha_+$.
	\end{remark}

	\begin{remark} \label{DDLocal}
	    We could have defined the local rank $DD(p,\varphi,k)$ as the supremum of depths of dividing sequences in
	      $p$ with all formulas  of the sequence equal to $\varphi$ and dividing with respect to $k$.   
	    But then, $DD(p,\varphi,k)$ equals the local rank $D(p,\varphi,k)$ (see for example 3.11 in
	      Casanovas\cite{Casanovas11}).
	\end{remark}

	\begin{remark}\label{FDEquiv}
		One can define the notions of forking sequence and forking chain in a similar way.
		The arguments in Proposition~\ref{DDEquiv} remain true to show the equivalence of items $4$, $5$ and $6$
		  replacing dividing by forking everywhere.
		Therefore it is natural to define a rank $FD$ as the set of possible lengths of forking sequences in $p$ 
		  (or equivalently forking chains). 
		However it is not so well behaved. There are two basic questions that remain open.
	\end{remark}

	\begin{question} 
	Does $FD$ depend on the set of parameters?
    	There is no similar tree equivalence of $FD$ showing that $FD$ does not depend on the set of parameters.
    	Therefore we denote it by $FD(p,A)$.
	\end{question}
			
    It is immediate that $DD(p)\leq FD(p,A)$ for any partial type $p$ over $A$.
    For finite values of $DD(p)$ is easy to prove by induction that $DD(p)=FD(p,A)$. 
    It is also easy to prove that if $DD(p)=\omega_-$, then $FD(p,A)=\omega_-$.
	In the next section we will see that for any partial type $p$ over $A$, $FD(p,A)=\infty$ if and only if $DD(p)=\infty$, but for the intermediate values is an open question if $DD=FD$

	\begin{question}
		Is $FD(p,A)=DD(p)$ for any partial type $p$?. 
		In that case, obviously $FD$ would not depend on the set of parameters.
	\end{question}

\section{Simple and Supersimple theories}
	Now, we are going to see that $DD$ characterize simple and supersimple theories. 
	\begin{lemma}\label{DDEqCardinals}
		Let $p(x)$ be a partial type over a set of parameters $A$.
		Let $\kappa$ be any regular cardinal number.
		The following are equivalent: 
		\begin{enumerate}
			\item $DD(p)<\kappa_+$.
			\item For every $B\supseteq A$ and $a\models p(x)$, there exists a set $B_0\subseteq B$ with
			  $|B_0|<\kappa$ such that $a\ind^d_{AB_0}B$.
		\end{enumerate}
	\end{lemma}
	\begin{proof} \

    	$1\Rightarrow 2$.
    		Assuming $2$ false, we build a dividing sequence of formulas with parameters in $B$ of depth $\kappa$ by
    		  recursion. 
    		Assume that we have a dividing sequence $(\varphi_i(x,y_i):i\in\alpha)$, $(b_i:i\in\alpha)$ with $a$
    		  realizing $\{\varphi_i(x,b_i):i\in\alpha\}$ and $\alpha<\kappa$.
    		Let $B_0=\{b_i:i\in\alpha\}$. 
    		Then $a\nind^d_{AB_0}B$ implies that there exists $\varphi_{\alpha}(x,y_{\alpha})$ with $b_{\alpha}\in B$
    		  such that $\models\varphi_{\alpha}(a,b_{\alpha})$ and $\varphi_{\alpha}(x,b_{\alpha})$ divides over
    		  $AB_0=Ab_{<\alpha}$.
    
    	$2\Rightarrow 1$.	
    	    Assume that $1$ is false and  let $(\varphi_i(x,y_i):i\in\kappa)$, $(b_i:i\in\kappa)$ be a dividing
    	      sequence in $p$ over $A$.
    		Let $a\models p(x)\cup\{\varphi_i(x,b_i):i\in\kappa\}$.
    		Let $B=A\{b_i:i\in\kappa\}$ and $B_0\subseteq B$ with $|B_0|<\kappa$.
    		As $\kappa$ is regular, there is some $i\in\kappa$ such that $B_0\subseteq Ab_{<i}$. 
    		As $\varphi_i(x,b_i)$ divides over $Ab_{<i}$, it also divides over  $AB_0$ and therefore
    		  $a\nind^d_{AB_0}B$.
	\end{proof}

	From this lemma we get:
	\begin{corollary}
		Let $T$ be a theory and $p$ a partial type.
		Then,
		\begin{enumerate}
			\item $T$ is simple if and only if $DD(x=x)<\infty$.
			\item $T$ is supersimple if and only if $DD(x=x)<\omega_+$.			
			\item $p$ is simple if and only if $DD(p)<\infty$.
		\end{enumerate}
	\end{corollary}
	\begin{proof}
        $1$ and $2$ are immediate from the definitions of simple and supersimple theories
          (Casanovas\cite{Casanovas11} or Kim\cite{Kim}).
        $3$ is immediate from the equivalences of the definition of a simple type in Chernikov\cite{Chernikov}.
	\end{proof}

	\begin{remark}
	    In C\'ardenas, Farr\'e\cite{CardenasFarre} is analyzed which is the natural definition of a supersimple type
	      and it is proved that $p$ supersimple is equivalent to $DD(p)<\omega_+$
	\end{remark}
	
	Following, a result about burden and strongness.
	\begin{remark}\label{bdnDD}
        Here, we compare $DD$-rank with the burden defined by Adler. 
        See Adler\cite{Adler} for the definitions of $inp$-pattern and burden ($bdn$).
        The existence of an $inp$-pattern of depth $\kappa$ in $p$ is equivalent to the existence of a tree of depth
          $\kappa$ in $p$ with the following restriction: $a_{s,i}=a_{t,i}$ for $|s|=|t|$.
	    Therefore if it exists an $inp$-pattern of depth $\kappa$ in $p$, then $\kappa_+\leq DD(p)$.
	    So, in general $bdn(p)\leq DD(p)$ and if it exists an $inp$-pattern of depth $bdp(p)$ in $p$, then 
	      $bdn(p)_+\leq DD(p)$ (we are taking $bdn(p)$ as an ordinal for the comparison).
	    From this, we deduce the next corollary and the known fact that a supersimple theory is strong.
	\end{remark}

	\begin{corollary} 
	    A partial type $p$ with $DD(p)<\omega_+$ (supersimple) is strong 
	      (there are not $inp$-pattern in $p$ of depth $\omega$).
	\end{corollary}

	It is equivalent that dividing chains and forking chains are not bounded:
	\begin{proposition}
		Let $p$ be a partial type over a set of parameters $A$.
		The following are equivalent:
		\begin{enumerate}
			\item $DD(p)<\infty$ ($p$ is simple).
			\item There is no forking chain of complete types in $p$ of depth $|T|^+$.
		\end{enumerate}
	\end{proposition}
	\begin{proof} 
		It is immediate that $2$ implies $1$. 
		In the other directions a standard proof works. See, for example, proposition 4.15 of
		  Casanovas\cite{Casanovas11}.
	\end{proof}

	\begin{remark}\label{forkingExistence}
		Let $p$ be a partial type over $A$.
		If $p$ is simple then $p$ does not fork over $A$.
		As any extension of $p$ is also simple we have the same for any partial type extending $p$.
	\end{remark}
	\begin{proof}
		If $p$ forks over $A$ then we can form a forking chain of partial types in $p$ of any length simply repeating
		  $p$ in the chain.
		So, if $p$ is simple then $p$ does not fork over $A$.
	\end{proof}
	
	Now, we are going to see that under certain hypothesis, non forking implies that $DD$ does not change.
	This is true in simple theories, but we will weaken the hypothesis.
	\begin{lemma}\label{DDeqInd}
      Let $p(x)$ be a partial type over $A$.
	  The following are equivalent:
	  \begin{enumerate}
        \item $\alpha_+\leq DD(p)$
		\item There exist $a\models p(x)$ and a sequence of parameters $(a_i:i\in\alpha)$ (with $a_i$ a finite tuple)
		  such that $a\nind^d_{Aa_{<i}}a_i$ for each $i\in\alpha$.
	  \end{enumerate}	
	  If $p(x)=tp(b/A)$, using an argument of conjugation we can assume $a=b$.
    \end{lemma}
	\begin{proof}
		Let $(p_i:i\in\alpha)$ and  $(A_i:i\in\alpha)$ be a dividing chain of partial types.
		Let $a\models p_{<\alpha}$.
		For each $i\in\alpha$ we can take $\varphi_i(x,a_i)$ that divides over $A_{<i}$ and
		  $p_i\vdash\varphi_i(x,a_i)$.
		We may assume that $a_i\in A_i$. 
		So for each $i\in\alpha$, $\models\varphi_i(a,a_i)$ and $\varphi_i(x,a_i)$ divides over $Aa_{<i}$.
		That is, for each $i\in\alpha, a\nind^d_{Aa_{<i}}a_i$.
		
		In the other direction, from $a\nind^d_{Aa_{<i}}a_i$ we obtain for each $i\in\alpha$ a formula
		  $\varphi_i(x,b_ia_i)$ with $b_i\subseteq Aa_{<i}$ that divides over $Aa_{<i}$ and
		  $\models\varphi_i(a,b_ia_i)$.
		Let $c_i=b_ia_i$.
		Then, $Ac_{<i}=Aa_{<i}$.
		Therefore, $(\varphi_i(x,c_i):i\in\alpha)$ is a dividing sequence in $p$.
	\end{proof}

	\begin{proposition}\label{DD_Dividing}
		Let $p(x)\in S(A)$ and $A\subseteq B$ with $tp(B/A)$ simple and co-simple. 
		Let $q(x)\supseteq p(x)$ be a partial type over $B$.
		If $q(x)$ does not fork over $A$ then $DD(q)=DD(p)$.
	\end{proposition}
	\begin{proof}
		We have to prove that $DD(p)\leq DD(q)$.
	    Let $\alpha_+\leq DD(p)$, so there exists $a\models p$ and a sequence $(a_i:i\in\alpha)$ such that for every
	      $i\in\alpha, a\nind^d_{Aa_{<i}}a_i$.
		Since $q$ does not fork over $A$, we can choose a completion $\bar{q}\in S(B)$ of $q$ such that $\bar{q}$
		  does not fork over $A$.
		Let $c\models\bar{q}$ and let $B'$ be such that $cB\equiv_A aB'$.
		
		As $tp(B/A)$ is simple, it is also simple $tp(B'/Aa)$. 
		By Remark~\ref{forkingExistence}, we can choose $B''$ such that $B''\equiv_{Aa}B'$ and
		  $B''\ind^f_{Aa}a_{<\alpha}$.
		Composing automorphisms we have $cB\equiv_A aB''$. 
		Let $(a'_i:i\in\alpha)$ be such that $cBa'_{<\alpha}\equiv_A aB''a_{<\alpha}$.
		So we have $B\ind^f_{Ac}a'_{<\alpha}$ and for every $i\in\alpha$, $c\nind^d_{Aa'_{<i}}a'_i$.

		By definition $6.7$ of Chernikov\cite{Chernikov}, as $tp(B/Ac)$ is co-simple we obtain
		  $a'_{<\alpha}\ind^f_{Ac}B$.
		As $\bar{q}$ does not fork over $A$, we have $c\ind^f_A B$, by left transitivity we have 
		$ca'_{<\alpha}\ind^f_A B$. 
		By definition $6.1$ of Chernikov\cite{Chernikov}, since  $tp(B/A)$ simple, $B\ind^f_A ca'_{<\alpha}$.
		Finally, $B\ind^f_{Aa'_{<i}}a'_i$ for every $i\in\alpha$.
		
		From this, we obtain for every $i\in\alpha, c\nind^d_{Ba'_{<i}}a'_i$.
		Otherwise $c\ind^d_{Ba'_{<i}}a'_i$ and $B\ind^d_{Aa'_{<i}}a'_i$ implies by left transitivity that
		  $Bc\ind^d_{Aa'_{<i}}a'_i$ and therefore immediately $c\ind^d_{Aa'_{<i}}a'_i$ which is contradictory.
		So we have $\alpha_+\leq DD(\bar{q})\leq DD(q)$.
	\end{proof}

    Now, we will see that given a type, we can get a subtype of bounded size with the same $DD$-rank and that given a
      simple type, we can get a simple subtype of bounded size.
    \begin{proposition}
        Let $p(x)$ be a partial type over $A$. Then,
        \begin{enumerate}
            \item If $p$ is simple, there exists a set of parameters $B\subseteq A$ such that $|B|\leq|T|$ and
              $p\upharpoonright B$ is simple.
            \item There exists a set of parameters $B\subseteq A$ such that $|B|\leq|T|^{|DD(p)|}$ and
              $DD(p\upharpoonright B)=DD(p)$.
        \end{enumerate}  
    \end{proposition}
    \begin{proof} 
        \emph{1}. 
            $p$ simple means that there is no $\varphi(x,y)$, $k<\omega$ and
              $(a_{s,i}: s\in\omega^{<\omega},i\in\omega)$ such that for each $s\in\omega^{<\omega}$,
              $\{\varphi(x,a_{s,i}):i\in\omega\}$ is $k$-inconsistent and such that for each
              $f:\omega\longrightarrow\omega$, $p(x)\cup\{\varphi(x,a_{f\upharpoonright i,f(i)}):i\in\omega\}$ is
              consistent.
            Then, we consider the type $\Sigma_{p,\varphi,k}$ in the variables
              $(x_{s,i}: s\in\omega^{<\omega},i\in\omega)$ that express the previous condition of consistency and
              $k$-inconsistency for $p$ and $\varphi$. 
            That is, $p$ is simple iff for every $\varphi(x,y)$ and $k<\omega$, $\Sigma_{p,\varphi,k}$ is
              inconsistent.
            As $p$ is simple, by compactness, for each $\Sigma_{p,\varphi,k}$ there is some finite
              $A_{\varphi,k}\subseteq A$ such that $\Sigma_{p\upharpoonright A_{\varphi,k},\varphi,k}$ is
              inconsistent. 
            Taking $B=\bigcup_{\varphi,k} A_{\varphi,k}$ we get $\Sigma_{p\upharpoonright B,\varphi,k}$ are
              inconsistent. 
            We are using that $p\subseteq q$ implies $\Sigma_{p,\varphi,k}\subseteq \Sigma_{q,\varphi,k}$.
                
        \emph{2}.  
            If $DD(p)=\infty$ is trivial, so we can assume $DD(p)<\infty$.
            In a similar way, given $\alpha$, for every sequence $\overline{\varphi}=(\varphi_j(x,y_j):j\in\alpha)$
              and every sequence $\overline{k}=(k_j:j\in\alpha)$ we can define types
              $\Sigma_{p,\alpha,\overline{\varphi},\overline{k}}$ whose inconsistency guarantees that
              $DD(p)\leq\alpha$. 
            So, in a similar way, we can obtain a subtype $q$ satisfying $DD(p)=DD(q)$.
    \end{proof}

	\begin{proposition}
		Let $p(x)$ be a partial type over $A$ with $DD(p)\geq\alpha_+\geq\omega_+$.
		\begin{enumerate}
            \item There exists $B\supseteq A$ and a completion $q\in S(B)$ of $p$ dividing over $A$ with $DD(q)\geq\alpha_+$.
		    \item If $M\supseteq A$ is an $|A|^+$-saturated model, there exists a completion $q\in S(M)$ of $p$ dividing over $A$ with $DD(q)\geq\alpha_+$.
		\end{enumerate}
	\end{proposition}
	\begin{proof}
		\emph{1} is immediate by equivalence $5$ of Proposition~\ref{DDEquiv}, taking the first type in the dividing chain.
		
		Let $M\supseteq A$ be an $|A|^+$-saturated model.
		We take a tree of parameters $(a_{s,i}:s<\omega^{<\alpha},i\in\omega)$, a sequence of formulas $\overline{\varphi}=(\varphi_j(x,y_j):j\in\alpha)$ and a sequence of numbers $\overline{k}=(k_j:j\in\alpha)$ witnessing $DD(p)\geq\alpha_+$, according to definition~\ref{DDDef}. As in the proof of the previous proposition, we can build a type $\Sigma_{p,\alpha,\overline{\varphi},\overline{k}}$ over $A$  in the variables $(z_{s,i}:s<\omega^{<\alpha},i\in\omega)$ whose consistency guarantees that $DD(p)\geq\alpha_+$.
		
		Denoting by $z=(z_{s,i}:s<\omega^{<\alpha},i\in\omega)-(z_{\emptyset,0})$, the type 
		  $\exists z\left(\Sigma_{p,\alpha,\overline{\varphi},\overline{k}}\right)$ has a realization in $M$, so we have some new parameters $(b_{s,i}:s<\omega^{<\alpha},i\in\omega)$ realizing the type
          $\Sigma_{p,\alpha,\overline{\varphi},\overline{k}}$, where $b_{\emptyset,0}\in M$.
        Taking a realization $a$ of $p$ and a tree branch including the $b_{\emptyset,0}$, we have $DD(a/Ab_{\emptyset,0})\geq\alpha_+$. By remark~\ref{DDsupCT} we can extend $tp(a/Ab_{\emptyset,0})$ to a type $q$ over $M$ with $DD(q)\geq\alpha_+$.
	\end{proof}

    \begin{corollary}
        Let $p(x)$ be a non-simple partial type over $A$. Then,
		\begin{enumerate}
            \item There exists $B\supseteq A$ and a completion $q\in S(B)$ of $p$ dividing over $A$ with $q$ non-simple.
		    \item If $M\supseteq A$ is an $|A|^+$-saturated model, there exists a completion $q\in S(M)$ of $p$ dividing over $A$ with $q$ non-simple.
		\end{enumerate}        
    \end{corollary}

\section{Additivity}
	We are going to see that the $DD$ rank under certain hypotheses, satisfies Lascar-style inequalities.
	In the next lemma $otp(S)$ denotes the order type of a well-ordered set $S$.
	That is, the ordinal order-isomorphic to $S$.
    
	\begin{remark}
		Let $p$ be a partial type over $A$ and $a\models p$.
		For any sequence $(a_i:i<\alpha)$ if we denote $\beta=otp\{i\in\alpha:a\nind^d_{Aa_{<i}}a_i\}$, 
		  then $\beta_+\leq DD(p)$.
	\end{remark}
	\begin{proof}
	  Let $f:\beta\longrightarrow\{i\in\alpha:a\nind^d_{Aa_{<i}}a_i\}$ be the order isomorphism.
      Let $b_i=a_{f(i)}$ for $i\in\beta$. As $Ab_{<i}\subseteq Aa_{<f(i)}$ and $a\nind^d_{Aa_{<f(i)}}a_{f(i)}$, 
        it holds  $a\nind^d_{Ab_{<i}}b_i$ for each $i\in\beta$.
	\end{proof}
  
	\begin{proposition}\label{DD_Add_Right}
		If $\alpha_+\leq DD(ab/A)$, then there exist ordinals $\beta$ and $\gamma$ such that 
		  $\alpha\leq\beta\oplus\gamma$ and $\beta_+\leq DD(a/A)$ and $\gamma_+\leq DD(b/Aa)$.
	\end{proposition}
	\begin{proof}
  	    Let $(a_i:i\in\alpha)$ such that $ab\nind^d_{Aa_{<i}}a_i$ for each $i\in\alpha$.
		We have $ab\nind^d_{Aa_{<i}}a_i$ implies $a\nind^d_{Aa_{<i}}a_i$ or $b\nind^d_{Aaa_{<i}}a_i$.
		Let $X=\{i\in\alpha:a\nind^d_{Aa_{<i}}a_i\}$, $Y=\{i\in\alpha:b\nind^d_{Aaa_{<i}}a_i\}$. 
		So, $\alpha=X\cup Y$ and therefore $\alpha\leq otp(X)\oplus otp(Y)$. 
		  By previous remark, $otp(X)_+\leq DD(a/A)$ and $otp(Y)_+\leq DD(b/Aa)$.
	\end{proof}

	\begin{definition}
		We recall that a theory is called \textbf{Extensible} if forking has existence, that is every complete type
		  does not fork over its parameters.
		For instance, simple theories are extensible.
	\end{definition}
	
	\begin{proposition}\label{DD_Add_Left}
		Let $T$ be a $NTP_2$ and extensible theory.
		If $\alpha_+\leq DD(a/A)$ and $\beta_+\leq DD(b/Aa)$, then $(\alpha+\beta)_+\leq DD(ab/A)$.
	\end{proposition}
	\begin{proof}
	    By Chernikov, Kaplan\cite{ChernikovKaplan}, in a $NTP_2$ and extensible theory, forking equals dividing,
	      so we use $\ind$ to denote independence without distinctions.
	    If $tp(ab/A)$ is not simple then $DD(ab/A)=\infty$ and the inequality is true. 
	    So we can suppose that $tp(ab/A)$ is simple and therefore, $tp(a/A)$ and $tp(b/Aa)$ are simple too.
	    As the theory is $NTP_2$, by theorem $6.17$ of Chernikov\cite{Chernikov}, they are also co-simple.
	      
	    Let $(a_i:i\in\alpha)$ such that $a\nind_{Aa_{<i}}a_i$ for every $i\in\alpha$ and let $(b_i:i\in\beta)$ such
	      that $b\nind_{Aab_{<i}}b_i$ for every $i\in\beta$.

		As forking has existence, we can choose $a'_{<\alpha}\equiv_{Aa}a_{<\alpha}$ such that verify 
		  $a'_{<\alpha}\ind_{Aa}b_{<\beta}$.
		We will check that $(a'_i:i\in\alpha)^{\frown}(b_i:i\in\beta)$ is a sequence of length $\alpha+\beta$ 
	      satisfying Lemma~\ref{DDeqInd} in $tp(ab/A)$.
		For the first part of the sequence: From $a\nind_{Aa_{<i}}a_i$ we obtain 
		  $a\nind_{Aa'_{<i}}a'_i$ and therefore, $ab\nind_{Aa'_{<i}}a'_i$.

		For the second part: From $a'_{<\alpha}\ind_{Aa}b_{<\beta}$, we have $a'_{<\alpha}\ind_{Aab_{<i}}b_i$.
		From  $b\nind_{Aab_{<i}}b_i$, by left transitivity, we obtain $b\nind_{Aaa'_{<\alpha}b_{<i}}b_i$.
		As $tp(b/Aa)$ is simple, $b_i\nind_{Aaa'_{<\alpha}b_{<i}}b$ and $b_i\nind_{Aa'_{<\alpha}b_{<i}}ab$.
		Finally, as $tp(ab/A)$ is co-simple, we have $ab\nind_{Aa'_{<\alpha}b_{<i}}b_i$.
	\end{proof}

	\begin{corollary}
		Let $T$ be a $NTP_2$ and extensible theory.
		If $DD(ab/A)$ is finite, then  
			\[DD(ab/A)=DD(a/A)+DD(b/Aa)\]
	\end{corollary}

	\begin{proposition}\label{DD_Add_Eq}\footnote{We thank Martin Ziegler for his suggestions in the proof of this result.}
		Let $T$ be a simple theory and let $a$ and $b$ such that $a\ind_Ab$.
		If $\alpha_+\leq DD(a/A)$ and $\beta_+\leq DD(b/Aa)$, then $(\alpha\oplus\beta)_+\leq DD(ab/A)$.
	\end{proposition}
	\begin{proof}
	    in a simple theory, forking equals dividing, so we use $\ind$ to denote independence without distinctions.
	    
	    Let $(a_i:i\in\alpha)$ such that $a\nind_{Aa_{<i}}a_i$ for every $i\in\alpha$ and let $(b_i:i\in\beta)$ such
	      that $b\nind_{Aab_{<i}}b_i$ for every $i\in\beta$.
		We can choose $a'_{<\alpha}\equiv_{Aa}a_{<\alpha}$ such that $a'_{<\alpha}\ind_{Aa}bb_{<\beta}$.
		Following the same steps that in Proposition~\ref{DD_Add_Left}, we can prove that
		  $ab\nind_{Aa'_{<\alpha}b_{<i}}b_i$ for every $i\in\beta$.  
	    If we also prove that $ab\nind_{Aa'_{<i}b_{<\beta}}a'_i$ for every $i\in\alpha$, we will get that any
	      shuffle of these parameters conserving the original order among the $a_i$ and the $b_i$ also works, 
	      showing  that $(\alpha\oplus\beta)_+\leq DD(ab/A)$.

        From $a\nind_{Aa_{<i}}a_i$ we obtain $a\nind_{Aa'_{<i}}a'_i$.
        From $a\ind_A b$ and $a'_{<i}\ind_{Aa}b$ we obtain $aa'_{<i}\ind_{A}b$ and therefore $a\ind_{Aa'_{<i}}b$.
        As $a\nind_{Aa'_{<i}}a'_i$ 
          we obtain $a\nind_{Aba'_{<i}}a'_i$.
        Since  
          $b_{<\beta}\ind_{Aba'_{<i}}a'_i$ we get $a\nind_{Aba'_{<i}b_{<\beta}}a'_i$.
        Finally, this implies $ab\nind_{Aa'_{<i}b_{<\beta}}a'_i$
	\end{proof}

	\begin{proposition}
		For every $a,b$ and $A$:
		\begin{enumerate}
			\item If $a\in acl(Ab)$, then $DD(ab/A)=DD(b/A)$.
			\item If $acl(aA)=acl(bA)$, then $DD(a/A)=DD(b/A)$.			
		\end{enumerate}
	\end{proposition}
	\begin{proof}
		We check $1$, $2$ follows from $1$.
		It is immediate that $DD(b/A)\geq\alpha_+$ implies $DD(ab/A)\geq\alpha_+$, so 
		  $DD(ab/A)\geq DD(b/A)$ is always true.
		If $a\in acl(Ab)$, then $DD(a/Ab)=0$ and we obtain the other direction by Proposition~\ref{DD_Add_Right}.
	\end{proof}

    We are going to express the content of \ref{DD_Add_Right}, \ref{DD_Add_Left} and \ref{DD_Add_Eq} with
      inequalities. 
    To do so, we introduce an arithmetic on the set of extended ordinals 
      $\{\alpha_{+} \, :\, \alpha\in Ord \}\cup\{\alpha_{-} \, :\, \alpha\in Lim\}$ defined in
      Definition~\ref{DDDef}:

    \begin{definition}
		We define $\hat{+}$:
		\begin{itemize}
			\item $\alpha_-\hat{+}\beta_-=\sup\{\gamma+\beta:\gamma<\alpha\}_-$
			\item $\alpha_-\hat{+}\beta_+=\sup\{\gamma+\beta:\gamma<\alpha\}$ \\
			  (with $_-$ or $_+$ depending if the supremum is reached).
			\item $\alpha_+\hat{+}\beta_-=(\alpha+\beta)_-$
			\item $\alpha_+\hat{+}\beta_+=(\alpha+\beta)_+$
		\end{itemize}
		And now define $\hat{\oplus}$:
		\begin{itemize}
			\item $\alpha_-\hat{\oplus}\beta_-=\sup\{\gamma\oplus\delta:\gamma<\alpha,\delta<\beta\}_-$
			\item $\alpha_+\hat{\oplus}\beta_-=(\alpha\oplus\beta)_-$
			\item $\alpha_+\hat{\oplus}\beta_+=(\alpha\oplus\beta)_+$
		\end{itemize}
    \end{definition}

    \begin{corollary}
        For every $a,b$ and $A$, 
        \begin{enumerate}
            \item $DD(ab/A)\leq DD(a/A)\hat{\oplus}DD(b/Aa)$.
            \item In a $NTP_2$ and extensible theory, $DD(a/A)\hat{+}DD(b/Aa)\leq DD(ab/A)$.
            \item In a simple theory, if $a\ind_Ab$, $DD(ab/A)=DD(a/A)\hat{\oplus}DD(b/Aa)$.
        \end{enumerate}
    \end{corollary}
	\begin{proof}
	    This is equivalent to \ref{DD_Add_Right}, \ref{DD_Add_Left} and \ref{DD_Add_Eq}.
	\end{proof}

	\begin{proposition}
		Let $S,T$ be type-definable sets.
		Then, 
		  \[DD(S)\hat{+}DD(T)\leq DD(S\times T)\leq DD(S)\hat{\oplus}DD(T)\]
		As $S$ and $T$ have the same role, at the left we can take the maximum between $DD(S)\hat{+}DD(T)$ and
		  $DD(T)\hat{+}DD(S)$.
	\end{proposition}
	\begin{proof}
		We can assume without loss of generality that $S$ and $T$ are type-defined over the same set of parameters
		  $A$.
        By Proposition~\ref{DD_Add_Right}, for any $a\in S, b\in T$, we have 
          $DD(ab/A)\leq DD(a/A)\hat{\oplus}DD(b/A)$, taking supremums and by Remark~\ref{DDsupCT} we obtain the right
          inequality.
			
		For the left inequality, we take a dividing sequence $\{\varphi_i(x,a_i):i\in\alpha\}$ for $S$ over $A$ and a
		  dividing sequence for $T$ over $Aa_{<\alpha}$. 
		The two sequences together form a dividing sequence for $S\times T$.
	\end{proof}


\begin{thebibliography}{99}
	\bibitem{Adler}			  H. Adler. Strong theories, burden and weight. Preprint, 2007.
	\bibitem{CardenasFarre}   S. C\'ardenas and R. Farr\'e. Foundation ranks and supersimplicity. Preprint.
	                            arXiv:2005.00520 [math.LO], 2020
	\bibitem{Casanovas11}     E. Casanovas. Simple theories and hyperimaginaries. Lecture Notes in Logic. 
	                          Cambridge University Press, 2011.
	\bibitem{Chernikov}		  A. Chernikov. Theories without the tree property of the second kind. 
	                          Annals of Pure and Applied Logic, 165:695-723, 2014.
	\bibitem{ChernikovKaplan} A. Chernikov and I. Kaplan. Forking and Dividing in NTP$_2$ theories. 
	                            arXiv:0906.2806v2 [math.LO], 2011.
	\bibitem{Kim}             B. Kim. Simplicity Theory. Oxford University Press, 2014.
	\bibitem{Ziegler}		  K. Tent and M. Ziegler. A Course in Model Theory. Lecture Notes in Logic.
	                          Cambridge University Press, 2012.
\end{thebibliography}
\end{document}